\documentclass[11pt]{amsart}
\usepackage{amsmath,amssymb,latexsym,soul,cite,mathrsfs}

\usepackage{color,enumitem,graphicx}
\usepackage[colorlinks=true,urlcolor=blue,
citecolor=red,linkcolor=blue,linktocpage,pdfpagelabels,
bookmarksnumbered,bookmarksopen]{hyperref}
\usepackage[english]{babel}

\usepackage[left=2.9cm,right=2.9cm,top=2.8cm,bottom=2.8cm]{geometry}
\usepackage[hyperpageref]{backref}

\usepackage[colorinlistoftodos]{todonotes}
\makeatletter
\providecommand\@dotsep{5}
\def\listtodoname{List of Todos}
\def\listoftodos{\@starttoc{tdo}\listtodoname}
\makeatother

\numberwithin{equation}{section}
\def\dis{\displaystyle}
\def\cal{\mathcal}
\newtheorem{lem}{Lemma}
\newtheorem{prop}{Proposition}
\newtheorem{theo}{Theorem}
\newtheorem{coro}{Corollary}
\newtheorem{rem}{Remark}

\DeclareMathOperator{\cat}{cat}

\title[Morse Theory for fractional Schr\"odinger equations]
{a Multiplicity result via Ljusternick-Schnirelmann category and morse theory for a fractional schr\"odinger equation in $\mathbb R^{N}$}

\author[G. M. Figueiredo]{Giovany M. Figueiredo}
\author[G. Siciliano]{Gaetano Siciliano}

\address[G. M. Figueiredo]{\newline\indent Faculdade de Matem\'atica
\newline\indent 
Universidade Federal do Par\'a
\newline\indent
66075-110, Bel\'em - PA, Brazil}
\email{\href{mailto:giovany@ufpa.br}{giovany@ufpa.br}}

\address[G. Siciliano]{\newline\indent Departamento de Matem\'atica
\newline\indent 
 Universidade de S\~ao Paulo 
\newline\indent 
Rua do Mat\~ao 1010,  05508-090 S\~ao Paulo, SP, Brazil }
\email{\href{mailto:sicilian@ime.usp.br}{sicilian@ime.usp.br}}

\thanks{Giovany M. Figueiredo was partially
supported by  CNPq/Brazil. Gaetano Siciliano  was partially supported by
Fapesp and CNPq, Brazil. }

\begin{document}

\maketitle

\begin{abstract}
In this work we study  the following class of
problems in $\mathbb R^{N}, N>2s$
$$
 \varepsilon^{2s} (-\Delta)^{s}u + V(z)u=f(u), \,\,\,
        u(z) > 0
$$
where $0<s<1$, $(-\Delta)^{s}$ is the fractional Laplacian, 
$\varepsilon$ is a positive parameter, the potential $V:\mathbb{R}^N \to\mathbb{R}$ 
and the nonlinearity $f:\mathbb R \to \mathbb R$ satisfy suitable assumptions;
in particular it is assumed that $V$ achieves its positive minimum on some set $M.$
By using variational methods we prove existence, multiplicity and
concentration of maxima
of positive solutions when $\varepsilon\to 0^{+}$.
In particular the multiplicity result
is obtained by means of the Ljusternick-Schnirelmann and Morse theory,
by exploiting the  ``topological complexity'' of the set $M$.
\end{abstract}

\section{Introduction}

In this paper we are concerned with  existence, multiplicity
and concentration results for the solutions of the following class of  problems
\begin{equation}\label{problema}\tag{$P_{\varepsilon}$}
\left\{
             \begin{array}{l}
              \varepsilon^{2s}(-\Delta)^{s}u + V(z)u=f(u)
               \ \ \text{in} \ \ \mathbb{R}^N, \quad N>2s\medskip\\ 
        u \in H^{s}(\mathbb R^N)\medskip
        \\
        u(z) > 0,  z \in \mathbb R^N ,
             \end{array}
           \right.
\end{equation}
where  $s \in (0,1)$, $\varepsilon > 0$ and  the Hilbert space $H^s(\mathbb R^N)$ is defined as
\[
H^s(\mathbb R^N)=\big\{u\in L^2(\mathbb R^N):\,
(-\Delta)^{s/2}u \in L^2(\mathbb R^N)\big\}
\]
endowed with scalar product and (squared) norm given by
\[
(u,v)= \int_{\mathbb R^{N}} (-\Delta)^{s/2}u (-\Delta)^{s/2}v + \int_{\mathbb R^{N}} uv,
\qquad
\|u\|^2=\|(-\Delta)^{s/2}u \|_2^2+ \|u\|_2^2.
\]
The fractional Laplacian $(-\Delta)^s$ is the pseudodifferential operator   defined via the Fourier transform
$$\mathcal F((-\Delta)^{s}u)=|\cdot|^{2s}\mathcal Fu,$$
and, when $u$ has sufficient regularity, it is also given by
\[
(-\Delta)^su(z)=-\frac{C(N,s)}{2}\int_{\mathbb R^{N}} \frac{u(z+y)-u(z-y)-2u(z)}{|y|^{N+2s}}dy, \quad z\in\mathbb R^N,
\]
where $C(N,s)$ is a suitable normalization constant.  For this fact and the relation between
the fractional Laplacian and the fractional Sobolev space $H^{s}(\mathbb R^{N})$
we refer the reader to classical books, see also  \cite{DPV}.

Problem \eqref{problema}  appears when one look for standing waves solutions 
$$\psi(z,t)=u(z)e^{-iE t/\varepsilon}, \quad u(z)\in \mathbb R,\ E \text{ a real constant}$$
to the following Fractional Schr\"odinger equation 
$$i\varepsilon\frac{\partial \psi}{\partial t}=\varepsilon^{2s}(-\Delta)^{s}\psi +W(z)\psi -f(|\psi|)$$
where $W: \mathbb R^{N}\to \mathbb R$ is an external potential and $f$ a suitable nonlinearity.
Here $\varepsilon$ is a sufficiently small parameter which corresponds to the Planck constant.

The fractional Schr\"odinger equation was first derived and studied by Laskin \cite{L1,L2,L3}.
After that many papers appeared studying existence, multiplicity and behavior of solutions
to fractional Schr\"odinger equations.
Recently in \cite{CZ}  the authors studied, by means of  Lyapunov-Schmidt reduction methods, concentration phenomenon 
 for solutions 
 in presence of a potential
 and with a power type nonlinearity. In particular it is shown that 
 for sufficiently small $\varepsilon$ the solutions concentrates to non-degenerate critical points
 of the potential. Concentration of solutions is also studied in \cite{SZ} where the authors consider
 the nonlinearity $f(x,u)=K(x)|u|^{p-2}u$ and prove the  concentration near suitable critical 
 points of a function $\Gamma (x)$ which involves the potential $V$ and the function $K.$
 We also mention  \cite{FMV} where it is shown that  concentration can occur only  at critical points of $V$.
 

\medskip

We recall also that in recent years, problems involving fractional  operators are receiving a special attention.
Indeed fractional spaces and nonlocal equations have important applications in many sciences.
We limit here ourself to give a non-exhaustive list of fields and papers in which these equations are used:
obstacle problem \cite{MS,S},
optimization and finance \cite{CT,DL}, phase transition \cite{ABS,SV}, 
material science \cite{B}, anomalous diffusion \cite{GM, MK1, MK2}, 
conformal geometry and minimal surfaces
\cite{CRS,CV,CG}.
The list may continue with applications in crystal dislocation, soft thin films, multiple scattering,
quasi-geostrophic flows, water waves, and so on. 
The interested reader may consult also the references in the cited papers.

\medskip

Coming back to our problem \eqref{problema}, in order to state the results we introduce the basics assumptions on $f$ and $V$:
\medskip

\begin{enumerate}[label=(V\arabic*),ref=V\arabic*]
\item\label{V} $V:\mathbb R^{N}\to \mathbb R$ is a continuous function and satisfies
$$
0<\min_{\mathbb R^N}V(x)=:V_{0}<\liminf_{|x|\rightarrow\infty}V(x)=: V_{\infty}\in (0,+\infty] \, ;
$$
\end{enumerate}

\begin{enumerate}[label=(f\arabic*),ref=f\arabic*,start=1]
\item\label{f_{1}} $f: \mathbb R \rightarrow \mathbb R $ is a function of class $C^{1}$ and
$f(u)=0$ for $u\leq0$; \medskip
\item\label{f_{2}} $\lim_{u \rightarrow 0} {f'(u)}=0$; \medskip 
\item\label{f_{3}} $ \exists \,q\in  (2, 2^{*}_{s}-1)$ such that $\lim_{u\rightarrow \infty}{f'(u)}/{u^{q-1}}=0$,
 where $2_{s}^{*}:=2N/(N-2s)$; \medskip
\item\label{f_{4}}  $ \exists\,  \theta>2$ such that $0<\theta F(u):=\theta \int^{u}_{0}f(t)dt
 \leq uf(u)$ for all $ u>0$; \medskip 
\item\label{f_{5}}  the function $ u\to f(u)/u$ is strictly increasing in $(0,+\infty).$
\end{enumerate}
\medskip

By a solution of \eqref{problema} we mean $u\in W_{\varepsilon}$ (see Section \ref{preliminari} for the definition of $W_{\varepsilon}$) such that
for every $ v\in W_{\varepsilon}$
$$ \varepsilon^{2s}\int_{\mathbb R^{N} } (-\Delta)^{s/2}u(-\Delta)^{s/2}v+\int_{\mathbb R^{N}}V(z)uv=\int_{\mathbb R^{N}}f(u)v$$
that is, as we will see, $u$ is a critical point of a suitable energy functional $I_{\varepsilon}$.
The solution with ``minimal energy''  is what we call a {\sl ground state}.

The assumptions  on $V$  and $f$ are quite natural in this context.
Assumption \eqref{V} was first introduced by Rabinowitz in \cite{Rabinowitz}
to take into account potentials which are possibly not coercive.
Hypothesis \eqref{f_{1}} is not restrictive since we are looking for positive solutions
(see e.g. \cite[pag. 1247]{FQT})
and \eqref{f_{2}}-\eqref{f_{5}} are useful to use variational teqniques
which involve  the Palais-Smale condition, the Mountain Pass Theorem and
the Nehari manifold.
To this aim we recall that $\{u_{n}\}$ is a Palais-Smale sequence for  a $C^{1}$ functional, let us say $I$,
at level $c\in \mathbb R$, if $I(u_{n})\to c$ and $I'(u_{n})\to 0.$
We will abbreviate this simply by saying that $\{u_{n}\}$ is a $(PS)_{c}$ sequence.
Moreover the functional $I$ is said to satisfy the Palais-Smale condition at level $c$,
if every $(PS)_{c}$ sequence has a (strongly) convergent subsequence.

Our first result concerns the existence of ground states solutions.

\begin{theo}\label{gio12}
Suppose that $f$ verifies \eqref{f_{1}}-\eqref{f_{5}} and $V$ verifies
\eqref{V}. Then there exists a ground state solution $\mathfrak u_{\varepsilon}\in W_{\varepsilon}$
of \eqref{problema}, 
\begin{itemize}
\item[1.] for every $\varepsilon>0$, if  $V_{\infty}=+\infty;$ \medskip
\item[2.] for every  $\varepsilon\in (0,\bar\varepsilon]$, for some $\bar\varepsilon>0$, if $V_{\infty}<+\infty$.
\end{itemize} 
\end{theo}

\medskip

The next results deal with the multiplicity of solutions and they involve
topological properties 
of the set of minima of the potential
$$
M:=\Big\{x\in\mathbb R^{N}:V(x)=V_{0}\Big\}.
$$
Indeed by means of the Ljusternik-Schnirelman theory 
we arrive at the following result.
\begin{theo}\label{gioGae1}
Suppose that $f$ satisfies \eqref{f_{1}}-\eqref{f_{5}} and the function $V$
satisfies \eqref{V}. Then, there exists $\varepsilon^{*}>0$ 
such that for every $\varepsilon\in (0,\varepsilon^{*}]$
problem \eqref{problema} has at least
\begin{itemize}
\item[1.] $\cat(M)$ positive solutions;\medskip
\item[2.] $\cat(M)+1$ positive solutions, if $M$ is bounded and $\cat (M)>1$.
\end{itemize}
Moreover,
for any such a solution $w_{\varepsilon}$, 
if  $\eta_{\varepsilon}\in \mathbb R^{N}$ denotes its
global maximum, it holds
$$
\lim_{\varepsilon\rightarrow 0^{+}}V(\eta_{\varepsilon})=V_{0}.
$$
\end{theo}

Hereafter $\cat_{Y}(X)$ denotes the Ljusternick-Schnirelmann category
of the set $X$ in $Y$ (if $X=Y$ we just write $\cat(X)$).
On the other hand, with the use of  Morse theory
we are able to deduce the next result.

\begin{theo}\label{gioGae2}
Suppose that $f$ satisfies \eqref{f_{1}}-\eqref{f_{5}} and the function $V$
satisfies \eqref{V}. Then there exists $\varepsilon^{*}>0$ such that for every $\varepsilon\in (0,\varepsilon^{*}]$ problem \eqref{problema}
 has at least $2\mathcal P_{1}(M)-1$ solutions,
if non-degenerate, possibly counted with their multiplicity.
\end{theo}
 We are denoting with  $\mathcal P_{t}(M)$ the Poincar\'e polynomial of $M$.
It is clear that in general, we get a better result using Morse theory; indeed, if for example $M$
is obtained  by  a contractible domain cutting off $k$ disjoint contractible sets, it is $\cat (M)=2$ 
and $\mathcal P_{1}(M)=1+k$. However, by using the Ljusternick-Schnirelmann
 category no non-degeneracy condition is required.

\begin{rem}
As it will be evident by the proofs, Theorem \ref{gio12} and Theorem \ref{gioGae1} remain true
if we replace conditions  \eqref{f_{2}} and \eqref{f_{3}} with the weaker conditions
\begin{itemize}
\item $\lim_{u \rightarrow 0} {f(u)/u}=0$; \smallskip 
\item $ \exists \,q\in  (2, 2^{*}_{s}-1)$ such that $\lim_{u\rightarrow \infty}{f(u)}/{u^{q}}=0$,
 where $2_{s}^{*}:=2N/(N-2s)$.
\end{itemize}
On the other hand for Theorem \ref{gioGae2} we need \eqref{f_{2}} and \eqref{f_{3}}
to have the compactness of a certain operator (see Section \ref{TH3}).

We have preferred to state our theorems under the stronger conditions just for the sake of simplicity.

\end{rem}

\medskip

The plan of the paper is the following. In Section \ref{preliminari},
after a change of variable, we introduce an equivalent problem to \eqref{problema}
and the related variational setting;
actually we will prove Theorems \ref{gio12}, \ref{gioGae1} and \ref{gioGae2} by referring
to this equivalent problem. In Section \ref{compattezza} we prove some compactness properties and give the proof of Theorem \ref{gio12}.
Section \ref{Bary} is devoted to introduce the barycenter map and its properties.
They will be fundamental tools in order to obtain the multiplicity results via the category theory
of Ljusternick-Schnirelmann, explored in Section \ref{TH2}, and the Morse theory given in Section \ref{TH3}.

\medskip

As a matter of notations, we denote with $B_{r}(y)$, respectively $B_{r}$, the ball in $\mathbb R^{N}$
 with radius $r>0$ centered in $y$, respectively in $0$. 
 The $L^{p}-$norm
in $\mathbb R^{N}$ is simply denoted with $|\cdot |_{p}$. If we  need
to specify the domain, let us say $A\subset \mathbb R^{N}$, we write $|\cdot |_{L^{p}(A)}$.


\section{Preliminaries and technical results}\label{preliminari}


First of all, it is easy to see that our problem 
 is equivalent, after a change of variable 
  to the following one
\begin{equation}\label{equivalente}\tag{$P_{\varepsilon}^{*}$}
\left\{
        \begin{array}{l}
           (-\Delta)^{s} u + V(\varepsilon x)u=f(u)
            \ \ \text{in} \ \ \mathbb R^N, \quad N>2s \smallskip
        \\
        u \in H^{s}(\mathbb R^N) \smallskip
        \\
        u(x) > 0, \  x \in \mathbb R^N 
             \end{array}
           \right.
\end{equation}
to which we will refer from now on.
Once we find solutions $u_{\varepsilon}$ for 
\eqref{equivalente}, 
the function $w_{\varepsilon}(x):=u_{\varepsilon}(x/\varepsilon)$ will be
a solution of \eqref{problema}.
Moreover, the maximum point
$\zeta_{\varepsilon}$ of $w_{\varepsilon}$ is related to the maximum point $z_{\varepsilon}$  of 
$u_{\varepsilon}$ simply by
$
\zeta_{\varepsilon}= \varepsilon z_{\varepsilon}.
$
Consequently, to prove the concentration property
 stated in Theorem \ref{gioGae1} we just need to show that
$$
\lim_{\varepsilon \to 0^{+}}V(\varepsilon z_{\varepsilon})=V_{0}.
$$

We fix now some notations involving the functionals
used to get the solutions to \eqref{equivalente}.

\subsection{ Variational setting}
Let us start with  the autonomous case. For a given constant (potential)
$\mu>0$ consider the problem
\begin{equation}\label{limite}\tag{$A_{\mu}$}
\left\{
             \begin{array}{l}
              (-\Delta)^{s} u + \mu u=f(u)
               \ \ \text{in} \ \ \mathbb R^N, \ \ N>2s \medskip
        \\
        u \in H^{s}(\mathbb R^N)  \medskip
        \\
        u(x) > 0, \  x \in \mathbb R^N 
             \end{array}
           \right.
\end{equation}
and the $C^{1}$ functional in $H^{s}(\mathbb R^{N})$
\begin{equation*}
E_{\mu}(u)=\frac{1}{2}\int_{\mathbb R^N}|(-\Delta)^{s/2} u|^2 +
          \frac{\mu}{2}\int_{\mathbb R^N} u ^2 -
          \int_{\mathbb R^N}F(u)
\end{equation*}
 whose critical points are the solutions of
\eqref{limite}.
In this case $H^{s}(\mathbb R^N)$ is endowed with the (squared) norm
$$\|u\|^{2}_{\mu} = \int_{\mathbb R^N}|(-\Delta)^{s/2} u|^{2} + \mu \int_{\mathbb R^N} u^{2}.$$

The following are well known facts. The functional $E_{\mu}$ has a mountain pass geometry
and, defining $\mathcal H=\{\gamma\in C([0,1], H^{s}(\mathbb R^{N})): \gamma(0)=0, E_{\mu}(\gamma(1))<0\}$,
 the mountain pass level 
 \begin{equation}\label{mmu}
 m(\mu):=\inf_{\gamma\in \mathcal H}\sup_{t\in[0,1]} E_{\mu}(\gamma(t))
 \end{equation}
satisfies
\begin{equation}\label{mplimite}
m(\mu)=\inf_{u\in H^{s}(\mathbb R^{N})\setminus\{0\}}\sup_{t\geq 0}E_{\mu}(tu)
=\inf_{u\in \mathcal M_{\mu}}E_{\mu}(u)>0,
\end{equation}
where
\begin{equation*}
{\cal M}_{\mu}:=\Big\{u\in H^{s}(\mathbb R^N)\setminus
\{0\}:\int_{\mathbb R^N}|(-\Delta)^{s/2} u|^2
+\mu\int_{\mathbb R^N}u^2=\int_{\mathbb R^N}f(u)u \Big\}.
\end{equation*}
It is standard to see that  $\mathcal M_{\mu}$ is bounded away from zero in $H^{s}(\mathbb R^{N}),$
and  is a differentiable manifold radially diffeomorphic to the unit sphere. It is usually called the {\sl Nehari manifold}
associated to $E_{\mu}$.

On the other hand, the solutions of
\eqref{equivalente} can be characterized as
 critical points of the $C^{1}$ functional given by
$$
I_{\varepsilon}(u) =\frac{1}{2}\int_{\mathbb R^N}|(-\Delta)^{s/2} u|^2 +
          \frac{1}{2}\int_{\mathbb R^N}V(\varepsilon x)u^2 -
          \int_{\mathbb R^N}F(u)
$$
which is well defined on the Hilbert  space 
$$
W_{\varepsilon}: = \left\{ u \in H^{s}({\mathbb R^N}) :
\int_{\mathbb R^N}V(\varepsilon x)u^{2} < \infty \right\}
$$
endowed with the (squared) norm
$$
\|u\|^{2}_{W_\varepsilon} = \int_{\mathbb R^N}|(-\Delta)^{s/2} u|^{2} +
\int_{\mathbb R^N}V(\varepsilon x)u^{2}.
$$
Note that if $V_{\infty}=+\infty, W_{\varepsilon}$ has compact embedding into $L^{p}(\mathbb R^{N})$
for $p\in [2,2_{s}^{*})$, see e.g. \cite[Lemma 3.2]{Cheng}.

The Nehari manifold associated to
$I_{\varepsilon}$ is
$$
{\cal N}_{\varepsilon}=\Big\{u\in W_{\varepsilon}\setminus
\{0\}: J_{\varepsilon}(u)=0 \Big\}
$$
where
\begin{eqnarray}\label{j}
J_{\varepsilon}(u):= \int_{\mathbb R^{N}}|(-\Delta)^{s/2} u|^{2} +
\int_{\mathbb R^{N}}V(\varepsilon x)u^{2} - \int_{\mathbb R^{N}}f(u)u
\end{eqnarray}
and its tangent space in $u$ is given by
$$\mathrm T_{u}\mathcal N_{\varepsilon}=\Big\{ v\in H^{s}(\mathbb R^{N}): J'(u)[v]=0\Big\}.$$

Let us introduce also 
$$\mathcal S_{\varepsilon}:=\Big\{u\in W_{\varepsilon}: \|u\|_{W_\varepsilon}=1, u>0 \ a.e.\Big\}$$
which is a smooth manifold of codimension 1.
The next result is standard; the proof follows the same lines of \cite[Lemma 2.1 and Lemma 2.2]{BC}.
\begin{lem} The following proposition hold true:
\begin{itemize}
\item[1.] for every $u\in \mathcal N_{\varepsilon}$ it is $J_{\varepsilon}'(u)[u]<0;$ \smallskip
\item[2.] $\mathcal N_{\varepsilon}$ is a differentiable manifold radially diffeormorphic to 
$\mathcal S_{\varepsilon}$ and there exists $k_{\varepsilon}>0$
such that $$\|u\|_{W_\varepsilon}\geq k_{\varepsilon}, \quad I_{\varepsilon}(u)\geq k_{\varepsilon}$$
\end{itemize}
\end{lem}
As in \cite[Lemma 2.1]{BC}, it is easy to see that the functions in $\mathcal N_{\varepsilon}$
have to be positive on some set of nonzero measure.
 It is also easy to check that  $I_{\varepsilon}$ has the mountain pass geometry, as given in the next
\begin{lem}\label{geometriadamontanha}
Fixed $\varepsilon>0,$ for the  functional $I_{\varepsilon}$ the following statements hold: \smallskip
\begin{itemize}
\item[i)] there exists $\alpha$, $\rho > 0$ such that
$
I_{\varepsilon}(u)\geq \alpha \,\,\, \mbox{with} \,\,\,
\|u\|_{\varepsilon}=\rho, \smallskip
$
\item[ii)] there exist $e\in W_{\varepsilon}$ with $\|e\|_{W_{\varepsilon}}>\rho$ such that  $I_{\varepsilon}(e)<0$.
\end{itemize}
\end{lem}
Then, defining the mountain pass level of $I_{\varepsilon}$,
\begin{equation*}\label{carat1}
c_{\varepsilon}:=\inf_{\gamma\in \mathcal H}\sup_{t\in[0,1]} I_{\varepsilon}(\gamma(t))
\end{equation*}
where $\mathcal H=\{\gamma\in C([0,1], W_{\varepsilon}): \gamma(0)=0, I_{\varepsilon}(\gamma(1))<0\}$,
 well known arguments 
  imply that 
 \begin{equation*}\label{carat2}
 c_{\varepsilon}=\inf_{u\in W_{\varepsilon}\setminus \{0\}}\sup_{t\geq
0}I_{\varepsilon}(tu)= \inf_{u\in {\cal
N}_{\varepsilon}}I_{\varepsilon}(u) \geq m(V_{0}).
 \end{equation*}

\medskip

\section{Compactness properties for $I_{\varepsilon}$ and $E_{\mu}$}\label{compattezza}

This section is devoted to prove compactness properties related to
the functionals $I_{\varepsilon}$ and $E_{\mu}$.

It is standard by now to see that hypothesis \eqref{f_{4}} is used to obtain the boundedness of the $PS$ sequences
for $I_{\varepsilon}$ or $E_{\mu}$:
we will always omit the prove of this fact in the paper.

We  need to recall the following Lions type lemma.
\begin{lem}\label{lionslemma}
If $\{u_{n}\}$ is bounded in $H^{s}(\mathbb R^{N})$ and for some $R>0$ and $2\leq r< 2_{s}^{*}$ we have
$$\sup_{x\in \mathbb R^{N}}\int_{B_R(x)}|u_{n}|^{r}\to 0\quad  \text{ as }\quad  n\to \infty,$$
then $u_{n}\to 0$ in $L^{p}(\mathbb R^{N})$ for $2< p <2_{s}^{*}$.
\end{lem}
For a proof see e.g. \cite[Lemma 2.3]{GPM}.

In order to prove compactness, some  preliminary work is needed.

\begin{lem}\label{gio5}
Let $\{u_{n}\}\subset W_{\varepsilon}$ be 
such that $I_{\varepsilon}'(u_{n})\to 0$
and $u_{n}\rightharpoonup 0$ in
$W_{\varepsilon}$.
Then we have either
\begin{itemize}
\item[a)]$u_{n} \rightarrow 0$ in $W_{\varepsilon}$, or \smallskip
\item[b)] there exist a sequence $\{y_{n}\} \subset {\mathbb R^N}$ and
constants $R, c > 0$ such that
$$
\liminf_{n\rightarrow +
\infty}\int_{B_{R}(y_{n})}u_{n}^{2} \geq c > 0.
$$
\end{itemize}
\end{lem}
\begin{proof}Suppose that b) does not occur. Using Lemma \ref{lionslemma}  it follows
$$
u_{n} \rightarrow 0 \,\,\, \mbox{in}\,\,\, L^{p}({\mathbb R^N})\,\,\,
\mbox{for}\,\,\, p \in (2, 2^{*}_{s}).
$$
Given $\xi > 0$, by \eqref{f_{2}} and \eqref{f_{3}}, for some constant $C_{\xi}>0$ we have
$$
0 \leq \int_{\mathbb R^N}f(u_{n})u_{n}\leq \xi \int_{\mathbb R^N}
u_{n}^{2} + C_{\xi} \int_{\mathbb R^N}|u_{n}|^{q + 1}.
$$
Using the fact that $\{u_{n}\}$ is bounded in $L^{2}(\mathbb R^{N})$,
$u_{n} \rightarrow 0$ in $L^{q+1}(\mathbb R^{N})$, and that $\xi$ is
arbitrary, we can conclude that
$$
\int_{\mathbb R^N}f(u_{n})u_{n}\rightarrow 0.
$$
Recalling that $\|u_{n}\|^{2}_{W_\varepsilon}-\int_{\mathbb R^N}f(u_{n})u_{n}= I'_{\varepsilon}(u_{n})[u_{n}]=o_{n}(1)$, it follows
that
$u_{n} \rightarrow 0$ in $ W_{\varepsilon}$. 
\end{proof}

\begin{lem}\label{gio9}
Assume that $V_{\infty}<\infty$ and let $\{v_{n}\}$ be a $(PS)_{d}$
sequence for  $I_{\varepsilon}$ in $W_{\varepsilon}$ with
$v_{n}\rightharpoonup 0$ in $W_{\varepsilon}$. Then
$$v_{n}\not\rightarrow 0\ \  \text{in}\ \ W_{\varepsilon}\ \Longrightarrow \ d \geq
m({V_{\infty}})$$
(recall that  $m({V_{\infty}})$ is the mountain pass level of
$E_{V_{\infty}}$, see \eqref{mplimite}).
\end{lem}
\begin{proof}
Let $\{t_{n}\}\subset (0,+\infty )$ be a
sequence such that $\{t_{n}v_{n}\}\subset {\cal M}_{V_{\infty}}$. We
start by showing the following

\medskip

\noindent \textbf{Claim} {\it The sequence $\{t_{n}\}$
satisfies $\limsup_{n\rightarrow \infty}t_{n} \leq 1$. }

\medskip
\noindent In fact, supposing by contradiction that the claim
does not hold, there exists $\delta
> 0$ and a subsequence still denoted by $\{t_{n}\}$, such that
\begin{eqnarray}\label{limsup}
t_{n} \geq 1 + \delta \quad \mbox{for all} \quad n \in \mathbb  N.
\end{eqnarray}
Since $\{v_{n}\}$ is bounded in $W_{\varepsilon}$,
$I'_{\varepsilon}(v_{n})[v_{n}]=o_{n}(1)$, that is,
$$
\int_{\mathbb R^N}\biggl[| (-\Delta)^{s/2} v_{n}|^{2} + V(\varepsilon
x)v_{n}^{2}\biggl] = \int_{\mathbb R^N}f(v_{n})v_{n} +
o_{n}(1).
$$
Moreover, since $\{t_{n}v_{n}\}\subset {\cal M}_{V_{\infty}}$, we get
$$
t_{n}^{2} \int_{\mathbb R^N}\biggl[| (-\Delta)^{s/2} v_{n}|^{2} +
V_{\infty} v_{n}^{2}\biggl]=
\int_{\mathbb R^N}f(t_{n}v_{n})t_{n}v_{n}.
$$
The last two equalities imply that
\begin{equation}\label{bo}
\int_{\mathbb R^N}\biggl[\frac{f(t_{n}v_{n})v_{n}^{2}}{t_{n}v_{n}}
- \frac{f(v_{n})v_{n}^{2}}{v_{n}}\biggl]=
\int_{\mathbb R^N}[V_{\infty} - V(\varepsilon x)]v_{n} ^{2}+
o_{n}(1).
\end{equation}
Given $\xi >0$, by condition \eqref{V} there exists $R=R(\xi)
> 0$ such that
$$
V(\varepsilon x) \geq V_{\infty} - \xi \quad \mbox{for any}
\,\,\,|x|\geq R.
$$
Let $C>0$ be such that $\|v_{n}\|_{W_\varepsilon}\leq C$. Since
$v_{n}\rightarrow 0$ in $L^{2}(B_{R}(0))$, we conclude by \eqref{bo}
\begin{eqnarray}\label{mudou}
\int_{\mathbb R^N}\biggl[\frac{f(t_{n}v_{n})}{t_{n}v_{n}} -
\frac{f(v_{n})}{v_{n}}\biggl]v_{n}^{2} \leq \xi CV_{\infty} +
o_{n}(1).
\end{eqnarray}
Since $v_{n}\not\rightarrow 0$ in $W_{\varepsilon}$, we may invoke
Lemma \ref{gio5} to obtain $\{y_{n}\}\subset \mathbb R^{N}$ and
$\check{R},{c} >0$ such that
\begin{eqnarray}\label{1eq4}
\int_{B_{\check R}( y_{n})} v_{n}^{2}\geq{c}.
\end{eqnarray}
Defining $\check{v}_{n} := v_{n}(\cdot+ y_{n})$, we may suppose
that, up to a subsequence,
$$
\check{v}_{n}\rightharpoonup \check{v} \,\,\, \mbox{in}\,\,\, H^{s}(\mathbb R^{N}).
$$
Moreover, in view of (\ref{1eq4}), there exists a subset $\Omega
\subset \mathbb R^{N}$ with positive measure such that $\check{v}>0$ in
$\Omega$. From \eqref{f_{5}}, we can use (\ref{limsup}) to rewrite
(\ref{mudou}) as
\begin{eqnarray*}
0 < \int_{\Omega}\biggl[\frac{f((1 + \delta)\check{v}_{n})}{(1+ \delta)\check{v}_{n}}
-\frac{f(\check{v}_{n})}{\check{v}_{n}}\biggl]\check{v}_{n}^{2} \leq \xi CV_{\infty} + o_{n}(1), \ \
\mbox{for any} \ \ \xi
> 0.
\end{eqnarray*}
Letting $n\rightarrow \infty$ in the last inequality and applying
Fatou's Lemma, it follows that
$$
0<\int_{\Omega}\biggl[\frac{f((1 + \delta)\check{v})}{(1 +\delta)\check{v}}
-\frac{f(\check{v})}{\check{v}}\biggl]\check{v}^{2}\leq
\xi CV_{\infty}, \ \ \mbox{for any} \ \ \xi>0.
$$
which is an absurd, proving the claim.

\medskip 

Now, it is convenient to distinguish the following cases: \medskip

\noindent \textbf{Case 1:} $ \limsup_{n \to \infty}t_{n}=1.$

In this case there exists a subsequence, still denoted by
$\{t_{n}\}$, such that $t_{n}\rightarrow 1$. Thus,
\begin{eqnarray}\label{1eq5}
d + o_{n}(1) = I_{\varepsilon}(v_{n}) \geq m({V_{\infty}}) +
I_{\varepsilon}(v_{n}) - E_{V_{\infty}}(t_{n}v_{n}).
\end{eqnarray}
Recalling that
\begin{eqnarray*}
I_{\varepsilon}(v_{n})- E_{V_{\infty}}(t_{n}v_{n})&=&
\frac{(1-t_{n}^{2})}{2}\int_{\mathbb R^{N}}|(-\Delta)^{s/2} v_{n}|^{2}+
\frac{1}{2}\int_{\mathbb R^{N}}V(\varepsilon x)v_{n}^{2} - \frac
{t_{n}^{2}}{2}\int_{\mathbb R^{N}}V_{\infty}v_{n}^{2}
\\ &+&
\int_{\mathbb R^{N}}[F(t_{n}v_{n}) - F(v_{n})],
\end{eqnarray*}
and using the fact that $\{v_{n}\}$ is bounded in $W_{\varepsilon}$ by $C>0$
together with the condition \eqref{V}, we get
\begin{eqnarray*}
I_{\varepsilon}(v_{n}) - E_{V_{\infty}}(t_{n}v_{n}) \geq \ o_{n}(1) -
C \xi + \int_{\mathbb R^{N}}[F(t_{n}v_{n}) - F(v_{n})].
\end{eqnarray*}
Moreover, by the Mean Value Theorem,
\begin{eqnarray*}
\int_{\mathbb R^{N}}[F(t_{n}v_{n}) - F(v_{n})]= o_{n}(1),
\end{eqnarray*}
therefore \eqref{1eq5} becomes
\begin{eqnarray*}
d + o_{n}(1) \geq m({V_{\infty}}) - C \xi + o_{n}(1),
\end{eqnarray*}
and taking the limit in $n$, by the arbitrariness of $\xi$,  we have
$d \geq m({V_{\infty}}).$ \medskip

\noindent \textbf{Case 2:} $\limsup_{n \to \infty}t_{n}=t_{0} <1$.

In this case up to a  subsequence,
still denoted by $\{t_{n}\}$, we have
$$
t_{n} \to t_{0} \,\,\, \mbox{and} \,\,\, t_{n} < 1 \,\,\, \text{ for all }
n \in \mathbb N .
$$
Since $u\mapsto \frac{1}{2}f(u)u-F(u)$ is increasing, we have
$$
m({V_{\infty}}) \leq
\int_{\mathbb R^{N}}\biggl[\frac{1}{2}f(t_{n}v_{n})t_{n}v_{n}-F(t_{n}v_{n})\biggl]
\leq\int_{\mathbb R^{N}}\biggl[\frac{1}{2}f(v_{n})v_{n}-F(v_{n})\biggl]\\
$$
hence,
$$
m({V_{\infty}}) \leq
I_{\varepsilon}(v_{n})-\frac{1}{2}I'_{\varepsilon}(v_{n})[v_{n}] = d +
o_{n}(1),
$$
and again we easily conclude.
\end{proof}

Now we are ready to give the desired compactness result.

\begin{prop}\label{gio11}
The functional $I_{\varepsilon}$  in $W_{\varepsilon}$ satisfies the $(PS)_{c}$ condition
\begin{itemize}
\item[1.] at any level $c < m({V_{\infty}})$, if $V_{\infty}<\infty$, \medskip
\item[2.] at any level $c\in\mathbb R$, if $V_{\infty}=\infty$.
\end{itemize}
\end{prop}
\begin{proof}
Let $\{u_{n}\}\subset W_{\varepsilon}$ be such
that $I_{\varepsilon}(u_{n})\rightarrow c$ and
$I'_{\varepsilon}(u_{n})\rightarrow 0$. By standard calculations, we
can see that $\{u_{n}\}$ is bounded in $W_{\varepsilon}$. Thus there
exists $u \in W_{\varepsilon}$ such that, up to a subsequence,
$u_{n}\rightharpoonup u$ in $W_{\varepsilon}$  
and we see that   $I'_{\varepsilon}(u)=0$.

Defining $v_{n}:= u_{n} - u$, by \cite{Alves3}
 we  know that $\int_{\mathbb R^{N}} F(v_{n})=\int_{\mathbb R^{N}}F(u_{n})-\int_{\mathbb R^{N}}F(u)+o(1)$
 and arguing as in
\cite{giovany} we have also 
$I'_{\varepsilon}(v_{n})\rightarrow 0$.
Then
\begin{equation}\label{cd}
I_{\varepsilon}(v_{n})=I_{\varepsilon}(u_{n})-I_{\varepsilon}(u)+
o_{n}(1)=c - I_{\varepsilon}(u)+o_{n}(1)=:d +o_{n}(1)
\end{equation}
and $\{v_{n}\}$ is a $(PS)_{d}$ sequence.
By \eqref{f_{4}},
\begin{eqnarray*}
I_{\varepsilon}(u)= I_{\varepsilon}(u)- \frac{1}{2}I'_{\varepsilon}(u)[u]=
\int_{\mathbb R^{N}}[\frac{1}{2}f(u)u - F(u)]\geq 0,
\end{eqnarray*}
and then, if $V_{\infty}<\infty$ and $c<m(V_{\infty})$, by \eqref{cd} we obtain
$$
d\leq c < m({V_{\infty}}).
$$
It follows from Lemma \ref{gio9} that $v_{n}\rightarrow 0$,
that is $u_{n}\rightarrow u$ in $W_{\varepsilon}$.

In the case  $V_{\infty}= \infty$ by 
the compact imbedding $W_{\varepsilon}
\hookrightarrow\hookrightarrow L^{p}(\mathbb R^{N}), 2 \leq p < 2^{*}_{s}$,
up to a subsequence, $v_{n}\rightarrow 0$ in
$L^{p}(\mathbb R^{N})$ and by \eqref{f_{2}} and \eqref{f_{3}}
\begin{eqnarray*}
\| v_{n}\|_{W_\varepsilon} ^{2}= \int_{\mathbb R^{N}}f(v_{n})v_{n}=o_{n}(1).
\end{eqnarray*}
This last equality implies that $u_{n}\rightarrow u$ in
$W_{\varepsilon}$.
\end{proof}

The next proposition is a direct consequence of the previous one, 
but for completeness we give the proof.

\begin{prop}\label{gio13}
The functional $I_{\varepsilon}$ restricted to $\cal{N}_{\varepsilon}$
satisfies the $(PS)_{c}$ condition 
\begin{itemize}
\item[1.] at any level $c <m({V_{\infty}})$, if $V_{\infty}<\infty$, \medskip
\item[2.]  at any level $c\in\mathbb R$, if $V_{\infty}=\infty$.
\end{itemize}
\end{prop}
\begin{proof}Let $\{u_{n}\}\subset \cal{N}_{\varepsilon}$ be such
that $I_{\varepsilon}(u_{n})\rightarrow c$ and 
for some sequence 
$\{\lambda_{n}\} \subset \mathbb R$,
\begin{eqnarray}\label{contra}
I'_{\varepsilon}(u_{n}) = \lambda_{n}J'_{\varepsilon}(u_{n}) + o_{n}(1),
\end{eqnarray}
where $J_{\varepsilon}:W_{\varepsilon}\rightarrow \mathbb R$ is defined in \eqref{j}.
Again we can deduce that $\{u_{n}\}$ is bounded. Now 
\begin{itemize}
\item[a)] evaluating \eqref{contra} in $u_{n}$ we get $\lambda_{n }J_{\varepsilon}'(u_{n})[u_{n}]=o_{n}(1)$, \medskip
\item[b)] evaluating \eqref{contra} in $v\in \mathrm{T}_{u_{n}}\mathcal N_{\varepsilon}$ we get $J_{\varepsilon}'(u_{n})[v]=0.$\medskip
\end{itemize}
Hence $\lambda_{n}J_{\varepsilon}'(u_{n})=o_{n}(1)$ and by \eqref{contra} we deduce $I_{\varepsilon}'(u_{n})=o_{n}(1)$.
Then $\{u_{n}\}$ is a $(PS)_{c}$ sequence for $I_{\varepsilon}$ and we conclude by  Proposition \ref{gio11}.
\end{proof}

\begin{coro}{\label{C1}} The constrained critical points of the functional $I_{\varepsilon}$ on
${\cal N}_{\varepsilon}$ are critical points of $I_{\varepsilon}$ in
$W_{\varepsilon}$.
\end{coro}
\begin{proof} The standard proof follows by using similar
arguments explored in the last proposition.\end{proof}

Now let us pass to the functional related to the autonomous problem \eqref{limite}.
\begin{lem}[Ground state for the autonomous problem]\label{gio16}
Let $\{u_{n}\} \subset {\cal{M}}_{\mu}$ be a sequence satisfying
$E_{\mu}(u_{n})\rightarrow m({\mu})$. Then, up to subsequences the following alternative holds:
\begin{itemize}
\item[a)] $\{u_{n}\}$ strongly converges in $H^{s}(\mathbb R^{N})$; \medskip
\item[b)] there exists a sequence $\{\tilde{y}_{n}\}\subset
\mathbb R^{N}$ such that $u_{n}(\cdot+\tilde{y}_{n})$ strongly converges  in $H^{s}(\mathbb R^{N})$.
\end{itemize}
In particular, there exists a minimizer $\mathfrak w_{\mu}\geq0$ for $m({\mu})$.
\end{lem}
This result is known in the literature, but for completeness
we give here the proof.
\begin{proof}By the Ekeland Variational Principle
we may suppose that $\{{u}_{n}\}$ is
a $(PS)_{m(\mu)}$ sequence for $E_{\mu}$. Thus going to a subsequence if
necessary, we have that $u_{n}\rightharpoonup u$ weakly in
$H^{s}(\mathbb R^{N})$ and it is easy to verify that  $E_{\mu}'(u)=0$.

In case $u\neq 0$, then $\mathfrak w_{\mu}:=u$ is a ground state
solution of the autonomous problem  \eqref{limite}, that is,
$E_{\mu}(\mathfrak w_{\mu})=m(\mu)$.

In case $u\equiv 0$, applying the same arguments employed in the proof
of Lemma \ref{gio5}, there exists a sequence $\{\tilde{y}_{n}\} \subset
\mathbb R^{N}$ such that
\begin{eqnarray*}
v_{n}\rightharpoonup v \ \ \text{in} \ \ H^{s}(\mathbb R^{N})
\end{eqnarray*}
where $v_{n}:={u}_{n}(\cdot + \tilde{y}_{n})$. Therefore, $\{v_{n}\}$ is also a
$(PS)_{m(\mu)}$ sequence of $E_{\mu}$ and $v\not\equiv 0$. It
follows from the above arguments that setting $\mathfrak w_{\mu}:=v$
it is the ground state solution we were looking for.

In both cases, it is easy to see that  $\mathfrak w_{\mu}\geq0$ and the proof of the
lemma is finished.
\end{proof}

\subsection{Proof of Theorem \ref{gio12}}

By Lemma \ref{geometriadamontanha}, the
functional $I_{\varepsilon}$ has the geometry of the Mountain Pass Theorem in $W_{\varepsilon}.$
Then by well known results
there exists $\{u_{n}\}\subset  W_{\varepsilon}$
satisfying
$$
I_{\varepsilon}(u_{n})\rightarrow c_{\varepsilon}\,\,\, \mbox{and}\,\,\,
I'_{\varepsilon}(u_{n})\rightarrow 0.
$$
\noindent {\bf  case I: $V_{\infty} = \infty$}. By Proposition \ref{gio11}, 
$\{u_{n}\}$ strongly converges to some $\mathfrak u_{\varepsilon}$
in $H^{s}(\mathbb R^{N})$, which satisfies
$$
I_{\varepsilon}(\mathfrak u_{\varepsilon})=c_{\varepsilon}\,\,\, \mbox{and}\,\,\,
I'_{\varepsilon}(\mathfrak u_{\varepsilon})=0.
$$
\noindent{\bf case II: $V_{\infty} < \infty$}. In virtue of  Proposition \ref{gio11} we just need to show that
$c_{\varepsilon} < m({V_{\infty}})$. 
Suppose without loss of generality that $0\in M$, i.e.
$$
V(0)=V_{0}.
$$
Let $\mu \in (V_{0}, V_{\infty})$, so that 
\begin{equation}\label{contradiz}
m({V_{0}})<m({\mu})<m({V_{\infty}}).
\end{equation}
For $r>0$ let $\eta_{r}$ a smooth cut-off function
in $\mathbb R^{N}$ which equals $1$ on $B_{r}$ and with support in $B_{2r}.$
Let $w_{r}:=\eta_{r} \mathfrak w_{\mu}$ and $t_{r}>0$ such that $t_{r}w_{r}\in \mathcal M_{\mu}.$
If it were, for every $r>0: E_{\mu}(t_{r}w_{r})\geq m(V_{\infty})$, since $w_{r}\to \mathfrak w_{\mu}$ in $H^{s}(\mathbb R^{N})$
for $r\to +\infty$, we would have $t_{r}\to 1$ and then
$$m(V_{\infty})\leq\liminf_{r\to+\infty}E_{\mu}(t_{r}w_{r})=E_{\mu}(\mathfrak w_{\mu})=m(\mu)$$
which contradicts \eqref{contradiz}. Then there exists $\overline r>0$ such that $\phi:= t_{\bar r} w_{\bar r}$
satisfies $E_{\mu}(\phi)<m(V_{\infty}).$
Condition \eqref{V} implies that
for some $\overline{\varepsilon}>0$
$$
V(\varepsilon x)\leq \mu,\,\,\, \mbox{for all}\,\,\,x \in
\mbox{supp}\, \phi \,\,\, \mbox{and}\,\,\, \varepsilon \leq
\overline{\varepsilon},
$$
so
\begin{eqnarray*}
\int_{\mathbb R^{N}}V(\varepsilon x) \phi ^{2} \leq
\mu \int_{\mathbb R^{N}} \phi ^{2}\,\,\, \mbox{for all}\,\,\,
\varepsilon \leq \overline{\varepsilon}
\end{eqnarray*}
and consequently
$$
I_{\varepsilon}(t \phi)\leq E_{\mu}(t \phi)\leq E_{\mu}(\phi)\,\,\, \mbox{for all}\,\,\, t> 0.
$$
Therefore
$
\max_{t>0}I_{\varepsilon}(t\phi)\leq E_{\mu}(\phi),
$
and  then $$c_{\varepsilon}<m(V_{\infty})$$
which conclude the proof.

%

\medskip



\section{The barycenter map}\label{Bary}

Up to now $\varepsilon$ was fixed in our considerations.
Now we deal with the case $\varepsilon\to 0^{+}.$
The next result will be fundamental when we implement
the ``barycenter machinery'' below.

\begin{prop}\label{gio17}
Let $\varepsilon_{n}\rightarrow 0$ and $\{u_{n}\}\subset
{\cal{N}}_{\varepsilon_{n}}$ be such that
$I_{\varepsilon_{n}}(u_{n})\rightarrow m({V_{0}})$. Then there exists
a sequence $\{\tilde{y}_{n}\}\subset \mathbb R^{N}$ such that
$u_{n}(\cdot+\tilde{y}_{n})$ has a convergent subsequence in
$H^{s}(\mathbb R^{N})$. Moreover, up to a subsequence,
$y_{n}:=\varepsilon_{n}\tilde{y}_{n}\rightarrow y \in M $.
\end{prop}
\begin{proof}Arguing as in the proof of Lemma \ref{gio5}, we
obtain a sequence $\{\tilde{y}_{n}\}\subset \mathbb R^{N}$ and constants
$R,c>0$ such that
\begin{eqnarray*}
\liminf_{n\rightarrow \infty}
\int_{B_{R}(\tilde{y}_{n})} u_{n}^{2}\geq c
> 0 .
\end{eqnarray*}
Thus, if $v_{n}:= u_{n}(\cdot + \tilde{y}_{n})$, up to a
subsequence, $v_{n}\rightharpoonup v \not\equiv 0$ in $H^{s}(\mathbb R^{N})$. Let $t_{n}>0$ be such that $\tilde{v}_{n}:=t_{n}v_{n}
\in {\cal{M}}_{V_{0}}$. Then,
$$
E_{V_{0}}(\tilde{v}_{n})\rightarrow m({V_{0}}).
$$
Since $\{t_{n}\}$ is bounded, so is the sequence $\{\tilde{v}_{n}\}$, thus for some subsequence,
$\tilde{v}_{n}\rightharpoonup \tilde{v}$  in $H^{s}(\mathbb R^{N})$. Moreover, reasoning as in \cite{giovany}, up to some
subsequence still denoted with $\{t_{n}\}$, we can assume that
$t_{n}\rightarrow t_{0}>0$, and this limit implies that
$\tilde{v}\not\equiv 0$. From Lemma \ref{gio16},
$\tilde{v}_{n}\rightarrow \tilde{v}$ in $H^{s}(\mathbb R^{N})$, and
so $v_{n}\rightarrow v$ in $H^{s}(\mathbb R^{N})$.

Now, we will show that $\{y_{n}\}:=\{\varepsilon_{n}\tilde{y}_{n}\}$ has a
subsequence verifying $y_{n}\rightarrow y \in M$.
First note that the sequence 
$\{y_{n}\}$ is bounded in $\mathbb R^{N}$.
 Indeed, assume by contradiction that (up to subsequences) $|y_{n}|\rightarrow \infty$. 

In case $V_{\infty}=\infty$, the  inequality
\begin{eqnarray*} \int_{\mathbb R^{N}}V(\varepsilon_{n} x + y_{n})
v_{n}^{2}\leq \int_{\mathbb R^{N}}| (-\Delta)^{s/2} v_{n}|^{2} +
\int_{\mathbb R^{N}}V(\varepsilon_{n} x + y_{n}) v_{n}^{2}=
\int_{\mathbb R^{N}}f(v_{n})v_{n},
\end{eqnarray*}
and the  Fatou's Lemma imply
\begin{eqnarray*}
\infty = \liminf_{n\rightarrow
\infty}\int_{\mathbb R^{N}}f(v_{n})v_{n}
\end{eqnarray*}
which is an absurd, since the sequence $\{f(v_{n})v_{n}\}$ is
bounded in $L^{1}(\mathbb R^{N})$.

Now let us consider the case $V_{\infty}< \infty$. Since
$\tilde{v}_{n}\rightarrow \tilde{v}$ in $H^{s}(\mathbb R^{N})$ and
$V_{0} < V_{\infty}$, we have
\begin{eqnarray*}
m({V_{0}})&=& \frac{1}{2}\int_{\mathbb R^{N}}| (-\Delta)^{s/2}
\tilde{v}|^{2}+ \frac{V_{0}}{2}\int_{\mathbb R^{N}}
\tilde{v}^{2}- \int_{\mathbb R^{N}}F(\tilde{v})\\
&<& \frac{1}{2}\int_{\mathbb R^{N}}| (-\Delta)^{s/2}\tilde{v}|^{2}+
\frac{V_{\infty}}{2}\int_{\mathbb R^{N}}\tilde{v}^{2}-
\int_{\mathbb R^{N}}F(\tilde{v})\\
&\leq& \liminf_{n\rightarrow
\infty}\biggl[\frac{1}{2}\int_{\mathbb R^{N}}| (-\Delta)^{s/2}
\tilde{v}_{n}|^{2}+ \frac{1}{2}\int_{\mathbb R^{N}}V(\varepsilon_{n}
x + y_{n}) \tilde{v}_{n}^{2}-
\int_{\mathbb R^{N}}F(\tilde{v}_{n})\biggl],
\end{eqnarray*}
\noindent or equivalently
\begin{eqnarray*}
m({V_{0}})< \liminf_{n\rightarrow
\infty}\biggl[\frac{t_{n}^{2}}{2}\int_{\mathbb R^{N}}| (-\Delta)^{s/2}
u_{n}|^{2}+ \frac{t_{n}^{2}}{2}\int_{\mathbb R^{N}}V(\varepsilon_{n}
z)u_{n}^{2}- \int_{\mathbb R^{N}}F(t_{n}u_{n})\biggl].
\end{eqnarray*}
The last inequality implies,
\begin{eqnarray*}
m({V_{0}})<  \liminf_{n\rightarrow
\infty}I_{\varepsilon_{n}}(t_{n}u_{n})\leq \liminf_{n\rightarrow
\infty}I_{\varepsilon_{n}}(u_{n})= m({V_{0}}),
\end{eqnarray*}
which is a contradiction. Hence, $\{y_{n}\}$ has to be  bounded and, up to
a subsequence, $y_{n}\rightarrow y\in\mathbb R^{N}$. If $y\not\in M$,
then $V(y)>V_{0}$ and we obtain a contradiction arguing as above.
Thus, $y\in M$ and the Proposition is proved. 
\end{proof}

\medskip

Let $\delta > 0$ be fixed 
and $\eta$ be a smooth nonincreasing
cut-off function defined in $[0,\infty)$ by
\begin{equation*}\label{eta}
\eta(s)=
\begin{cases}
1 & \mbox{ if } 0 \leq s \leq \delta/2\\
0 & \mbox{ if } s\geq \delta.
\end{cases}
\end{equation*}

Let $\mathfrak w_{V_{0}}$ be a ground state solution
given in Lemma \ref{gio16} of
problem \eqref{limite} with $\mu=V_{0}$ and
for any $y \in M$, let us define
\begin{eqnarray*}
\Psi_{\varepsilon , y }(x) := \eta (|\varepsilon x - y|)\mathfrak w_{V_{0}}\biggl(\frac{\varepsilon x -y}{\varepsilon}\biggl).
\end{eqnarray*}
Let $t_{\varepsilon}>0$ verifying
$\max_{t\geq 0}I_{\varepsilon}(t\Psi_{\varepsilon
,y})=I_{\varepsilon}(t_{\varepsilon}\Psi_{\varepsilon ,y}),
$
so that $t_{\varepsilon}\Psi_{\varepsilon,y}\in \mathcal N_{\varepsilon}$,
and let 
\begin{equation*}\label{Phi}
\Phi_{\varepsilon}: y\in M\mapsto  t_{\varepsilon}\Psi_{\varepsilon ,y}\in \cal{N}_{\varepsilon}.
\end{equation*}
By construction, $\Phi_{\varepsilon}(y)$ has compact support for any
$y\in M$ and  $\Phi_{\varepsilon}$ is a continuous map.

\medskip

The next result will help us to define a map from $M$
to a suitable sublevel in the Nehari manifold.

\begin{lem}\label{gio14}
The function $\Phi_{\varepsilon}$ satisfyes
\begin{eqnarray*}
\lim_{\varepsilon \rightarrow 0^{+}}I_{\varepsilon}(\Phi_{\varepsilon}(
y))=m({V_{0}}), \ \text{uniformly} \ in \ y \in M.
\end{eqnarray*}
\end{lem}
\begin{proof}Suppose by contradiction that the lemma is false.
Then there exist $\delta_{0}> 0$, $\{y_{n}\} \subset M$ and
$\varepsilon_{n}\rightarrow 0^{+}$ such that
\begin{eqnarray}\label{1eq7}
| I_{\varepsilon_{n}}(\Phi_{\varepsilon_{n}}(y_{n}))- m({V_{0}})|
\geq \delta _{0}.
\end{eqnarray}
Repeating the same arguments explored in \cite{giovany} (see also
\cite{Alves}), it is possible to check that
$t_{\varepsilon_{n}}\rightarrow 1$. From Lebesgue's Theorem, we can
check that
\begin{eqnarray*}\label{1eq9}
\lim_{n\rightarrow
\infty}\|\Psi_{\varepsilon_{n},y_{n}}\|^{2}_{\varepsilon_{n}}
= \| \mathfrak w_{V_{0}}\|^{2}_{V_{0}}
\end{eqnarray*}
and
\begin{eqnarray*}
\lim_{n\rightarrow
\infty}\int_{\mathbb R^{N}}F(\Psi_{\varepsilon_{n},y_{n}})=\int_{\mathbb R^{N}}F(\mathfrak w_{V_{0}}).
\end{eqnarray*}
Now, note that
\begin{eqnarray*}
I_{\varepsilon_{n}}(\Phi_{\varepsilon_{n}}(y_{n}))&=&
\frac{t^{2}_{\varepsilon_{n}}}{2}\int_{\mathbb R^{N}}\Big|(-\Delta)^{s/2}
(\eta(|\varepsilon_{n}z|)\mathfrak w_{V_{0}}(z))\Big|^{2} + \frac{t^{2}_{\varepsilon
_{n}}}{2}\int_{\mathbb R^{N}}V(\varepsilon_{n}z + y_{n})|\eta(|
\varepsilon_{n}z|)\mathfrak w_{V_{0}}(z)|^{2} \\
&-& \int_{\mathbb R^{N}}F(t_{\varepsilon
_{n}}\eta(|\varepsilon_{n}z|)\mathfrak w_{V_{0}}(z)).
\end{eqnarray*}
Letting $n\rightarrow\infty$, we get
$\lim_{n\rightarrow\infty}I_{\varepsilon_{n}}(\Phi_{\varepsilon_{n}}(y_{n}))=E_{V_{0}}(\mathfrak w_{V_{0}})=
m({V_{0}})$, which contradicts (\ref{1eq7}). Thus the Lemma holds.
\end{proof}

\bigskip

Observe  that by Lemma \ref{gio14}, $h(\varepsilon):=|I_{\varepsilon}(\Phi_{\varepsilon}(y))-m(V_{0})|=o(1)$
for $\varepsilon\to 0^{+}$ uniformly in $y$,
and then $I_{\varepsilon}(\Phi_{\varepsilon}(y))-m(V_{0})\leq h(\varepsilon)$.
In particular the set
\begin{equation}\label{subnehari}
{\cal{N}}^{m(V_{0})+h(\varepsilon)}_{\varepsilon}:=\Big\{u\in{\cal{N}}_{\varepsilon}:I_{\varepsilon}(u)\leq
m({V_{0}})+h(\varepsilon)\Big\}
\end{equation}
is not empty, since 
 for sufficiently small $\varepsilon$,
\begin{equation}\label{phiepsilon}
\forall\, y\in M: \Phi_{\varepsilon}(y)\in {\cal{N}}^{m(V_{0})+h(\varepsilon)}_{\varepsilon}.
\end{equation}

\bigskip

We are in a position now to define the barycenter map
that will send a convenient sublevel in the Nehari manifold
in a suitable neighborhood of $M$.
From now on we fix a  $\delta >0$ in such a way that
$M$ and 
$$M_{2\delta}:=\Big\{x\in \mathbb R^{N}: d(x,M)\leq 2\delta\Big\}$$ are homotopically equivalent
($d$ denotes the euclidean distance).
Let $\rho=\rho(\delta) > 0$ be such that
$M_{2\delta}\subset B_{\rho}$ and  $\chi : \mathbb R^{N}\rightarrow
\mathbb R^{N}$ be defined as
\begin{equation*}\label{chi}
\chi(x)= 
\begin{cases}
x & \mbox{ if } | x | \leq \rho\\
\rho \dis\frac{x}{|x|} & \mbox{ if } | x | \geq \rho.
\end{cases}
\end{equation*}
 Finally, let us consider the so called {\sl barycenter map}
$\beta_{\varepsilon}$ defined on functions with compact support $u\in W_{\varepsilon}$ by 
$$
\beta_{\varepsilon}(u):=\frac{\dis\int_{\mathbb R^{N}}\chi(\varepsilon x)
u^{2}(x)}{\dis\int_{\mathbb R^{N}}u^{2}(x)}\in \mathbb R^{N}.
$$
\begin{lem}\label{gio15} The function $\beta_{\varepsilon}$ satisfies
\begin{eqnarray*}
\lim_{\varepsilon \rightarrow 0^{+}}\beta_{\varepsilon}(\Phi_{\varepsilon}(y)) = y, \ \
\text{uniformly in }  y \in M.
\end{eqnarray*}
\end{lem}
\begin{proof} Suppose, by contradiction, that the lemma is
false. Then, there exist $\delta_{0}> 0$, $\{y_{n}\} \subset M$ and
$\varepsilon_{n}\rightarrow 0^{+}$ such that
\begin{eqnarray}\label{1eq11}
| \beta_{\varepsilon_{n}}(\Phi_{\varepsilon_{n}}(y_{n})) - y_{n}| \geq
\delta_{0}.
\end{eqnarray}
Using the definition of $\Phi_{\varepsilon_{n}}(y_{n}),\beta_{\varepsilon_{n}}$ and $\eta$ given above,
we have the  equality
\begin{eqnarray*}
\beta_{\varepsilon_{n}}(\Phi_{\varepsilon_{n}}(y_{n}))= y_{n} + \frac{\dis\int_{\mathbb R^{N}}
[\chi(\varepsilon_{n}z +
y_{n})-y_{n}]\Big|\eta(|\varepsilon_{n}z|)w(z)\Big|^{2}}
{\dis \int_{\mathbb R^{N}}\Big|\eta(|\varepsilon_{n}z|)w(z)\Big|^{2}}.
\end{eqnarray*}
Using the fact that $\{y_{n}\}\subset M\subset B_{\rho} $ and the
Lebesgue's Theorem, it follows
\begin{eqnarray*}
| \beta_{\varepsilon_{n}}(\Phi_{\varepsilon_{n}}(y_{n})) - y_{n}| = o_{n}(1),
\end{eqnarray*}
which contradicts (\ref{1eq11}) and the Lemma is proved.
\end{proof}

\vspace*{.4cm}

\begin{lem}\label{gio18}

We have
\begin{eqnarray*}
\lim_{\varepsilon \rightarrow 0^{+}}\ \sup_{u \in
\cal{N}^{m(V_{0})+h(\varepsilon)}_{\varepsilon}} \ \inf_{y \in
M_{\delta}}\Big| \beta_{\varepsilon} (u)- y \Big|  = 0.
\end{eqnarray*}
\end{lem}
\begin{proof}Let $\{\varepsilon_{n}\}$ be such that
$\varepsilon_{n}\rightarrow 0^{+}$. For each $n \in \mathbb N$, there exists
$u_{n}\in  \cal{N}^{m(V_{0})+h(\varepsilon_{n})}_{{\varepsilon}_{n}}$ such that
$$
\inf_{y \in M_{\delta}}\Big| \beta_{\varepsilon_{n}} (u_{n})- y \Big|
= \sup_{u \in \cal{N}^{m(V_{0})+h(\varepsilon_{n})}_{\varepsilon_{n}}} \inf_{y
\in M_{\delta}}\Big|\beta_{\varepsilon_{n}} (u)- y \Big|  +  o_{n}(1).
$$
Thus, it suffices to find a sequence $\{y_{n}\}\subset M_{\delta}$
such that
\begin{eqnarray}\label{1eq18}
\lim_{n\rightarrow  \infty}\biggl|\beta_{\varepsilon_{n}}(u_{n})-y_{n}\biggl| \ \
= 0.
\end{eqnarray}
Recalling  that $u_{n}\in
\cal{N}^{m(V_{0})+h(\varepsilon_{n})}_{\varepsilon_{n}}\subset{\cal{N}}_{\varepsilon_{n}}$
we have,
\begin{eqnarray*}\label{1eq19}
m({V_{0}}) \leq c_{\varepsilon_{n}}\leq I_{\varepsilon_{n}}(u_{n})\leq
m({V_{0}}) + h(\varepsilon_{n}),
\end{eqnarray*}
 so
$ I_{\varepsilon_{n}}(u_{n})\rightarrow m({V_{0}}).
$
By Proposition \ref{gio17}, we get a sequence $\{\tilde{y}_{n}\}\subset \mathbb R^N$
such that $v_{n}:={u}_{n}(\cdot+\tilde{y}_{n})$ converges in $H^{s}(\mathbb R^{N})$
to some $v$ and  $\{y_{n}\}:=\{\varepsilon_{n}\tilde{y}_{n}\}\subset M_{\delta}$, for $n$ sufficiently large. Thus
$$
\beta_{\varepsilon_{n}}(u_{n})= y_{n}+\frac{\dis\int_{\mathbb R^{N}}[\chi (\varepsilon_{n}z +
y_{n})- y_{n}] v_{n}^{2}(z)}{\dis\int_{\mathbb R^{N}}
v_{n}(z)^{2}}.
$$
Since $v_{n} \to
v$ in $H^{s}(\mathbb R^{N})$, it is easy to check that the sequence
$\{y_{n}\}$ verifies (\ref{1eq18}).
\end{proof}

\medskip

In virtue of  Lemma \ref{gio18},
 there exists $\varepsilon^{*}>0$ such that 
$$\forall\,\varepsilon\in (0,\varepsilon^{*}]: \ \ \sup_{u\in {\mathcal N}^{m(V_{0})+h(\varepsilon)}_{\varepsilon}} d(\beta_{\varepsilon}(u), M_{\delta})<\delta/2.$$

Define now
$$M^{+}:=\Big\{x\in \mathbb R^{N}: d(x,M)\leq 3\delta/2\Big\}$$
so that  $M$ and $M^{+}$ are homotopically equivalent.

Now,  reducing  $\varepsilon^{*}>0$ if necessary,
we can assume that Lemma \ref{gio15},  Lemma \ref{gio18} and \eqref{phiepsilon} hold.
Then by standard arguments the composed map

\begin{equation}\label{fundamental}
M\stackrel{\Phi_{\varepsilon}}{\longrightarrow} {\cal{N}}^{m(V_{0})+h(\varepsilon)}_{\varepsilon}\stackrel{\beta_{\varepsilon}}{\longrightarrow}M^{+}
\quad \text{ is homotopic to the inclusion map.}
\end{equation}

In case $V_{\infty}<\infty$,
we eventually reduce $\varepsilon^{*}$ in such a way that
also the Palais-Smale condition is satisfied in the interval  $(m(V_{0}), m(V_{0})+h(\varepsilon))$,
see Proposition \ref{gio13}.

\section{Proof of Theorem \ref{gioGae1}}\label{TH2}
\subsection{Existence}
By \eqref{fundamental} and well known properties of the category,
we get
\begin{eqnarray*}
\cat({\cal{N}}^{m(V_{0})+h(\varepsilon)}_{\varepsilon})\geq
\cat_{M^{+}}(M),
\end{eqnarray*}
and the Ljusternik-Schnirelman theory  (see e.g.
\cite{ghoussoub}) implies that
$I_{\varepsilon}$ has at least $\cat_{M^{+}}(M)=\cat(M)$ critical points
on ${\cal{N}}_{\varepsilon}$. 

To obtain another solution we use the same ideas of  \cite{BC}.
First note that, since $M$ is not contractible, 
the set $\mathcal A:={\Phi_{\varepsilon}(M)}$
can not be contractible in ${\cal{N}}^{m(V_{0})+h(\varepsilon)}_{\varepsilon}$. 
Moreover $\mathcal A$ is compact. 

For $u\in W_{\varepsilon}\setminus\{0\}$
we denote with $t_{\varepsilon}(u)>0$ the unique positive number such that $t_{\varepsilon}(u) u\in \mathcal N_{\varepsilon}.$
Let $u^{*}\in W_{\varepsilon}$ be such that $u^{*}\geq 0$, and $I_{\varepsilon}(t_{\varepsilon}(u^{*})u^{*})>m(V_{0})+h(\varepsilon).$
Consider the cone
$$\mathfrak C:=\Big\{tu^{*}+(1-t)u: t\in [0,1], u\in\mathcal A \Big\}$$
and note that $0\notin \mathfrak C$, since functions in $\mathfrak C$ have to be positive on a set of nonzero measure.
Clearly it is compact and contractible.
Let 
$$t_{\varepsilon}(\mathfrak C):=\Big\{t_{\varepsilon}(w)w: w\in \mathfrak C\Big\}$$
be its projection on $\mathcal N_{\varepsilon}$, which is compact  as well, and 
$$c:=\max_{t_{\varepsilon}(\mathfrak C)}I_{\varepsilon}>m(V_{0})+h(\varepsilon).$$
Since $\mathcal A\subset t_{\varepsilon}(\mathfrak C)\subset \mathcal N_{\varepsilon}$
and $t_{\varepsilon}(\mathfrak C)$ is contractible in $\mathcal N^{c}_{\varepsilon}:=\{u\in \mathcal N_{\varepsilon}: I_{\varepsilon}(u)\leq c\}$,
we infer that also $\mathcal A$ is contractible in $\mathcal N^{c}_{\varepsilon}$.

Summing up, we have a set $\mathcal A$ which is contractible in $\mathcal N^{c}_{\varepsilon}$
but not  in $\mathcal N^{m(V_{0})+h(\varepsilon)}_{\varepsilon}$, where $c>m(V_{0})+h(\varepsilon).$
This is only possible, since $I_{\varepsilon}$ satisfies the Palais-Smale condition, if there is
a critical level  between $m(V_{0})+h(\varepsilon)$ and $c$.

By Corollary \ref{C1}, we conclude the proof of  statements about the existence of solutions
in Theorem \ref{gioGae1}.

\medskip

\subsection{Concentration of the maximum points.}
The next two lemmas play a role in the study of the behavior of the
maximum points of the solutions. In the proof of the next lemma,
we adapted some arguments found in \cite{Gongbao}, which are
related with the Moser iteration method \cite{Moser}.

\begin{lem}\label{leminha1}
Assume  the conditions \eqref{V} and \eqref{f_{1}}-\eqref{f_{5}}. Let 
 $v_{n}\in H^{s}(\mathbb R^{N})$ be such that 
$$
\left\{
             \begin{array}{l}
              (-\Delta)^{s}v_{n} + V_{n}(x)v_{n}=f(v_{n})
               \ \ \text{ in } \ \ \mathbb R^N, \quad N>2s \\
        v_{n}(x) > 0,  \  x \in \mathbb R^N ,
             \end{array}
           \right.
$$
where  $V_{n}(x):=V(\varepsilon_{n}x+\varepsilon_{n}\tilde{y}_{n})$, and suppose that
 $v_{n}\rightarrow v$ in $H^{s}(\mathbb R^{N})$ with $v\not\equiv
0$. Then $v_{n} \in L^{\infty}(\mathbb R^{N})$ and there exists $C>0$
such that $|v_{n}|_{\infty}\leq C$ for all $n \in
\mathbb N$. Furthermore
$$
\lim_{|x|\rightarrow \infty}v_{n}(x)=0 \,\,\, \mbox{uniformly in
n}.
$$
\end{lem}
\begin{proof}For any $R>0$, $0<r\leq {R}/{2}$, let $\eta
\in C^{\infty}(\mathbb R^{N})$, $0\leq \eta\leq 1$ with $\eta(x)=1$ if
$|x|\geq R$ and $\eta(x)=0$ if $|x|\leq R-r$ and $|(-\Delta)^{s/2}
\eta|\leq {2}/{r}$. Note that by \eqref{f_{3}} we obtain the
following growth condition for $f$:
\begin{eqnarray}\label{1moser}
f(u)\leq \xi|u|+C_{\xi}|u|^{2^{*}_{s}-1}.
\end{eqnarray}

For each $n \in \mathbb N$ and for $L
> 0$, define
\begin{eqnarray*}
v_{L,n}(x) =
\begin{cases}
v_{n}(x)  &\text{ if } \quad v_{n}(x) \leq L\\
L \quad  &\text{ if }\quad v_{n}(x) \geq L,     
\end{cases}
\end{eqnarray*}
\begin{eqnarray*}
z_{L,n} := \eta^{2}v_{L,n}^{2(\sigma - 1)}v_{n} \quad \text{and} \quad
w_{L,n} := \eta v_{n} v_{L,n}^{\sigma - 1}
\end{eqnarray*}
with $\sigma > 1$ to be determined later.

 Taking $z_{L,n}$ as a test function, we obtain
\begin{eqnarray*}
\int_{\mathbb R^N}\eta^{2}v_{L,n}^{2(\sigma -1)}|(-\Delta)^{s/2} v_{n}|^{2} &=&
-2(\beta -1)\int_{\mathbb R^N}v_{L,n}^{2\sigma -
1}\eta^{2}v_{n}|(-\Delta)^{s/2} v_{n} (-\Delta)^{s/2} v_{L,n}\\
&+& \int_{\mathbb R^N}f(v_{n})\eta^{2}v_{n}v_{L,n}^{2(\sigma - 1)} -
\int_{\mathbb R^N}V_{n}v_{n}^{2}\eta^{2}v_{L,n}^{2(\sigma - 1)}\\
&-&2\int_{\mathbb R^N}\eta v_{L,n}^{2(\sigma -1)}v_{n}(-\Delta)^{s/2} v_{n} (-\Delta)^{s/2} \eta.
\end{eqnarray*}
By (\ref{1moser}) and for a $\xi$ sufficiently small, we have the
following  inequality
\begin{eqnarray*}
\int_{\mathbb R^N}\eta^{2}v_{L,n}^{2(\sigma -1)}|(-\Delta)^{s/2} v_{n}|^{2}
\leq C_{\xi}\int_{\mathbb R^N}v_{n}^{2^{*}_{s}}\eta^{2}v_{L,n}^{2(\sigma -
1)}-2\int_{\mathbb R^N}\eta v_{L,n}^{2(\sigma -1)}v_{n}|(-\Delta)^{s/2}v_{n}(-\Delta)^{s/2} \eta.
\end{eqnarray*}

\noindent For each $\varepsilon >0$, using the Young's inequality we
get
\begin{eqnarray*}
\int_{\mathbb R^N}\eta^{2}v_{L,n}^{2(\sigma -1)}|(-\Delta)^{s/2} v_{n}|^{2}
&\leq& C_{\xi}\int_{\mathbb R^N}v_{n}^{2^{*}_{s}}\eta^{2}v_{L,n}^{2(\sigma
- 1)}+2\varepsilon \int_{\mathbb R^N}\eta^{2}v_{L,n}^{2(\sigma -1)}|(-\Delta)^{s/2}v_{n}|^{2}\\
&+&2C_{\varepsilon}\int_{\mathbb R^N}v_{n}^{2}v_{L,n}^{2(\sigma -
1)}|(-\Delta)^{s/2} \eta|^{2}.
\end{eqnarray*}
Choosing $\varepsilon >0$ sufficiently small,
\begin{eqnarray}\label{2moser}
\int_{\mathbb R^N}\eta^{2}v_{L,n}^{2(\sigma -1)}|(-\Delta)^{s/2} v_{n}|^{2}
\leq C\int_{\mathbb R^N}v_{n}^{2^{*}_{s}}\eta^{2}v_{L,n}^{2(\sigma -
1)}+C\int_{\mathbb R^N}v_{n}^{2}v_{L,n}^{2(\beta - 1)}|(-\Delta)^{s/2}
\eta|^{2}.
\end{eqnarray}

\noindent Now, from Sobolev imbedding and Holder inequalities
\begin{eqnarray}\label{3moser}
|w_{L,n}|^{2}_{2^{*}_{s}} \leq
C\beta^{2}\biggl[\int_{\mathbb R^N}v_{n}^{2}v_{L,n}^{2(\sigma -
1)}|(-\Delta)^{s/2} \eta|^{2} + \int_{\mathbb R^N}\eta^{2} v_{L,n}^{2(\sigma -
1)}|(-\Delta)^{s/2} v_{n}|^{2}\biggl].
\end{eqnarray}

\noindent Using (\ref{2moser}) in (\ref{3moser}), we have
\begin{eqnarray}\label{4moser}
|w_{L,n}|^{2}_{2^{*}} \leq
C\sigma^{2}\biggl[\int_{\mathbb R^N}v_{n}^{2}v_{L,n}^{2(\sigma -
1)}|(-\Delta)^{s/2} \eta|^{2}
+\int_{\mathbb R^N}v_{n}^{2^{*}_{s}}\eta^{2}v_{L,n}^{2(\sigma -
1)}\biggl].
\end{eqnarray}

\noindent We claim that $v_{n} \in L^{{2_{s}^{*^{2}}}/{2}}(\mathbb R^{N}\setminus B_{R})$
for $R$ large enough and uniformly in $n$. In fact, let
$\sigma= {2^{*}_{s}}/{2}$. From (\ref{4moser}), we have
\begin{eqnarray*}
|w_{L,n}|^{2}_{2^{*}} \leq
C\sigma^{2}\biggl[\int_{\mathbb R^N}v_{n}^{2}v_{L,n}^{2^{*}-2}|(-\Delta)^{s/2}
\eta|^{2}
+\int_{\mathbb R^N}v_{n}^{2^{*}_{s}}\eta^{2}v_{L,n}^{2^{*}_{s}-2}\biggl]
\end{eqnarray*}
or equivalently
\begin{eqnarray*}
|w_{L,n}|^{2}_{2^{*}_{s}} \leq
C\sigma^{2}\biggl[\int_{\mathbb R^N}v_{n}^{2}v_{L,n}^{2^{*}_{s}-2}|(-\Delta)^{s/2}
\eta|^{2}
+\int_{\mathbb R^N}v_{n}^{2}\eta^{2}v_{L,n}^{2_{s}^{*}-2}v_{n}^{2_{s}^{*}-2}\biggl].
\end{eqnarray*}

\noindent Using the H\"{o}lder inequality with exponent
${2^{*}_{s}}/{2}$ and ${2^{*}_{s}}/{(2_{s}^{*}-2)}$
\begin{eqnarray*}
|w_{L,n}|^{2}_{2^{*}} \leq
C\sigma^{2}\int_{\mathbb R^N}v_{n}^{2}v_{L,n}^{2^{*}_{s}-2}|(-\Delta)^{s/2}
\eta|^{2}+C\sigma^{2}\biggl(\int_{\mathbb R^N}\biggl[v_{n}\eta
v_{L,n}^{(2^{*}_{s}-2)/2}\biggl]^{2^{*}}\biggl)^{2/2^{*}}
\biggl(\int_{\mathbb R^{N}\setminus B_{R/2}}v_{n}^{2^{*}}\biggl)^{(2^{*}_{s}-2)/2^{*}_{s}}.
\end{eqnarray*}

 From the definition of $w_{L,n}$ we have
\begin{eqnarray*}
\biggl(\int_{\mathbb R^N}\biggl[v_{n}\eta
v_{L,n}^{(2^{*}_{s}-2)/2}\biggl]^{2^{*}_{s}}\biggl)^{2/2^{*}_{s}}&\leq&
C\sigma^{2}\int_{\mathbb R^N}v_{n}^{2}v_{L,n}^{2^{*}_{s}-2}|(-\Delta)^{s/2}
\eta|^{2}\\
&+&C\sigma^{2}\biggl(\int_{\mathbb R^N}\biggl[v_{n}\eta
v_{L,n}^{(2^{*}_{s}-2)/2}\biggl]^{2^{*}_{s}}\biggl)^{2/2^{*}_{s}}
\biggl(\int_{\mathbb R^{N}\setminus B_{R/2}}v_{n}^{2^{*}_{s}}\biggl)^{2^{*}_{s}-2/2_{s}^{*}}.
\end{eqnarray*}

Since $v_{n}\rightarrow v$ in $H^{s}(\mathbb R^{N})$, for $R$
sufficiently large, we conclude
\begin{eqnarray*}
\int_{\mathbb R^{N}\setminus B_{R/2}}v_{n}^{2^{*}_{s}} \leq \varepsilon \,\,\,
\mbox{uniformly in $n$.}
\end{eqnarray*}
Hence
\begin{eqnarray*}
\biggl(\int_{\mathbb R^{N}\setminus B_{R}}\biggl[v_{n}
v_{L,n}^{(2^{*}_{s}-2)/2}\biggl]^{2^{*}_{s}}\biggl)^{2/2^{*}_{s}}\leq
C\sigma^{2}\int_{\mathbb R^N}v_{n}^{2}v_{L,n}^{2^{*}_{s}-2}
\end{eqnarray*}
or equivalently
\begin{eqnarray*}
\biggl(\int_{\mathbb R^{N}\setminus B_{R}}\biggl[v_{n}
v_{L,n}^{(2^{*}_{s}-2)/2}\biggl]^{2^{*}}\biggl)^{2/2_{s}^{*}}\leq
C\sigma^{2}\int_{\mathbb R^N}v_{n}^{2^{*}_{s}}\leq K < \infty .
\end{eqnarray*}
 Using the Fatou's lemma in the variable $L$, we have
\begin{eqnarray*}
\int_{\mathbb R^{N}\setminus B_{R}}v_{n}^{2_{s}^{*^{2}}/2} < \infty
\end{eqnarray*}
and therefore the claim holds.

Next, we note that if $\sigma=2^{*}_{s}{(t-1)}/{2t}$ with
$t={2_{s}^{*^{2}}}/{2(2^{*}_{s}-2)}$, then $\sigma >1$,
${2t}/{(t-1)}<2^{*}_{s}$ and $v_{n} \in L^{\sigma 2t/t-1}(\mathbb R^{N}\setminus B_{R-r})$.

Returning to inequality (\ref{4moser}), we obtain
\begin{eqnarray*}
|w_{L,n}|^{2}_{2^{*}_{s}} \leq C\sigma^{2}\biggl[\int_{ B_{R}\setminus B_{R-r}}v_{n}^{2}v_{L,n}^{2(\sigma - 1)} 
+\int_{\mathbb R^{N}\setminus B_{R-r}}v_{n}^{2_{s}^{*}}v_{L,n}^{2(\sigma - 1)}\biggl]
\end{eqnarray*}
or equivalently
\begin{eqnarray*}
|w_{L,n}|^{2}_{2^{*}_{s}} \leq C\sigma^{2}\biggl[\int_{B_{R}\setminus B_{R-r}}v_{n}^{2\sigma}
+\int_{\mathbb R^{N}\setminus B_{R-r}}v_{n}^{2^{*}_{s}-2}v_{n}^{2\sigma }\biggl].
\end{eqnarray*}

\noindent Using the H\"{o}lder's inequality with exponent
$t/(t-1)$ and $t$, we get
\begin{eqnarray*}
|w_{L,n}|^{2}_{2^{*}_{s}} &\leq&
C\sigma^{2}\biggl\{\biggl[\int_{B_{R}\setminus B_{R-r}}v_{n}^{2\sigma t/(t-1)}\biggl]^{(t-1)/t}\biggl[\int_{B_{R}\setminus B_{R-r}}1\biggl]^{1/t}\\
&+&\biggl[\int_{\mathbb R^{N}\setminus B_{R-r}}v_{n}^{(2^{*}_{s}-2)t}\biggl]^{1/t}\biggl[\int_{\mathbb R^{N}\setminus B_{R-r}}v_{n}^{2\sigma t/(t-1)}\biggl]^{t/(t-1)}\biggl\}.
\end{eqnarray*}

\noindent Since that $(2^{*}_{s}-2)t=2_{s}^{*^{2}}$, we conclude
\begin{eqnarray*}
|w_{L,n}|^{2}_{2^{*}_{s}} \leq C\sigma^{2}\biggl(\int_{\mathbb R^{N}\setminus B_{R-r}}v_{n}^{2\sigma t/(t-1)}\biggl)^{(t-1)/t}.
\end{eqnarray*}

\noindent Note that
\begin{eqnarray*}
|v_{L,n}|^{2\sigma}_{L^{2^{*}_{s}\sigma}(\mathbb R^{N}\setminus B_{R})} &\leq&
\biggl(\int_{\mathbb R^{N}\setminus B_{R-r}}v_{L,n}^{2^{*}_{s}\sigma
}\biggl)^{2/2_{s}^*}\leq \biggl(\int_{\mathbb R^{N}}\eta^{2}
v_{n}^{2^{*}_{s}}v_{L,n}^{2^{*}(\sigma -1) }\biggl)^{2/2_{s}^*}\\
&=&|w_{L,n}|^{2}_{2^{*}_{s}} \leq C\sigma^{2}\biggl
(\int_{\mathbb R^{N}\setminus B_{R-r}}v_{n}^{2\sigma t/(t-1)}\biggl)^{(t-1)/t}\\
&=& C\sigma^{2}|v_{n}|^{2\sigma}_{L^{2\sigma t /(t-1)}(\mathbb R^{N}\setminus B_{R-r})}.
\end{eqnarray*}

\noindent Applying Fatou's lemma
\begin{eqnarray*}
|v_{n}|^{2\sigma}_{L^{2^{*}_{s}\sigma}(\mathbb R^{N}\setminus B_{R})} \leq
C\sigma^{2}|v_{n}|^{2\sigma}_{L^{2\sigma t /(t-1)}(\mathbb R^{N}\setminus B_{R-r})}.
\end{eqnarray*}

\noindent Considering $\chi = {2^{*}_{s}(t-1)}/{2t}$,
$\zeta=2t/(t-1)$ and the last inequality, we can prove that
\begin{eqnarray*}
|v_{n}|_{L^{\chi^{m+1} \zeta}(\mathbb R^{N}\setminus B_{R})} \leq
C^{\sum_{i=1}^{m}\chi^{-i}}\chi^{\sum_{i=1}^{m}i\chi^{-i}}
|v_{n}|_{L^{2^{*}_{s}}(\mathbb R^{N}\setminus B_{R-r})},
\end{eqnarray*}
which implies
\begin{eqnarray*}
|v_{n}|_{L^{\infty}(\mathbb R^{N}\setminus B_{R})}\leq C|v_{n}|_{L^{2^{*}_{s}}(\mathbb R^{N}\setminus B_{R-r})}.
\end{eqnarray*}
Using again the convergence of $\{v_{n}\}$ to $v$ in
$H^{s}(I\!\!R^{N})$, for $\xi>0$ fixed there exists $R>0$
such that
\begin{eqnarray*}
|v_{n}|_{L^{\infty}(\mathbb R^{N}\setminus B_{R})}< \xi \quad \text{ for all } n \in \mathbb N .
\end{eqnarray*}

\noindent Thus,
$$
\lim_{|x|\rightarrow \infty}v_{n}(x)=0 \,\,\, \mbox{uniformly in }
n
$$
and the proof of the Lemma is finished. 
\end{proof}

Finally we have

\begin{lem}\label{leminha2}
There exists $\delta>0$ such that $|v_{n}|_{\infty} \geq \delta,
$ for every $ n \in \mathbb N. $

\end{lem}

\begin{proof}
Suppose that $|v_{n}|_{\infty} \to
0$.
It follows by
\eqref{f_{5}} that there exists $n_{0} \in \mathbb N$ such that,
$$
\frac{f(|v_{n}|_{\infty})}{|v_{n}|_{\infty}}<\frac{V_{0}}{2},
\,\,\, \mbox{for} \,\,\, n \geq n_{0}.
$$
Hence
$$
\int_{\mathbb R^{N}}|(-\Delta)^{s/2}
v_{n}|^{2}+\int_{\mathbb R^{N}}V_{0}v_{n}^{2}\leq
\int_{\mathbb R^{N}}\frac{f(|v_{n}|_{\infty})}{|v_{n}|_{\infty}}v_{n}^{2}
\leq \frac{V_{0}}{2}\int_{\mathbb R^N}v_{n}^{2},
$$
thus $\|v_{n}\|_{V_{0}} = 0$ for $n \geq n_{0}$,
which is an absurd, because $v_{n}\not=0$ for every  $n \in \mathbb N$. 
\end{proof}

\medskip

For what concerns the behavior of the maximum points when $\varepsilon\to0^{+},$
let $u_{\varepsilon_{n}}$ be a solution of problem $(P_{\varepsilon_{n}})$. Then
$v_{n}(x)=u_{\varepsilon_{n}}(x+\tilde{y}_{n})\in H^{s}(\mathbb R^{N})$ is a solution of
$$
  \left\{
             \begin{array}{l}
              (-\Delta)^{s} v_{n} + V_{n}(x)v_{n}=f(v_{n})
               \ \ in \ \ \mathbb R^N
        \\
        v_{n}(x) > 0,  \,x \in \mathbb R^N ,
             \end{array}
           \right.
$$
with $V_{n}(x):=V(\varepsilon_{n}x+\varepsilon_{n}\tilde{y}_{n})$ and
$\{\tilde{y}_{n}\}\subset \mathbb R^{N}$ are those given in Proposition \ref{gio17}.
Moreover, up to a subsequence, $v_{n}\rightarrow v$ in
$H^{s}(\mathbb R^{N})$ and $y_{n}\rightarrow y$ in $M$, where
$y_{n}=\varepsilon_{n}\tilde{y}_{n}$. 
By Lemma \ref{leminha1} and Lemma \ref{leminha2},  the global
maxima $p_{n}$ of $v_{n}$ are all in 
$ B_{R}$ for some $R>0$. Thus, the global
maximum of $u_{\varepsilon_{n}}$ is
$z_{\varepsilon}=p_{n}+\tilde{y}_{n}$ and therefore
$$
\varepsilon_{n}z_{\varepsilon_{n}}=\varepsilon_{n}p_{n}+\varepsilon_{n}\tilde{y}_{n}=\varepsilon_{n}p_{n}+y_{n}.
$$
Since $\{p_{n}\}$ is bounded, we have
\begin{equation*}\label{fine}
\lim_{n\rightarrow
\infty}V(\varepsilon_{n}z_{\varepsilon_{n}})=V_{0}.
\end{equation*}
We conclude the proof of Theorem \ref{gioGae1} 
in virtue of the considerations made at the beginning of Section \ref{preliminari}.

\section{Proof of Theorem \ref{gioGae2}}\label{TH3}
Before prove the theorem we first recall some basic facts of Morse theory
and fix some notations.

For a pair of topological spaces $(X,Y)$,
$Y\subset X,$ let $H_{*}(X,Y)$ be its singular homology with coefficients in some field $\mathbb F$
(from now on omitted) and 
$$\mathcal P_{t}(X,Y)=\sum_{k}\dim H_{k}(X,Y)t^{k}$$
the Poincar\'e polynomial of the pair. If $Y=\emptyset$, it will be always omitted in the objects which involve the pair.
Recall that  if $H$ is an Hilbert space,  $I:H\to \mathbb R$  a $C^{2}$ functional  and 
$u$ an isolated critical point  with $I(u)=c$, the  {\sl polynomial Morse index} of $u$ is
$$\mathcal I_{t}(u)=\sum_{k}\dim C_{k}(I,u)  t^{k}$$
where $C_{k}(I,u)=H_{k}(I^{c}\cap U, (I^{c}\setminus\{u\})\cap U)$ are the critical groups.
Here $I^{c}=\{u\in H: I(u)\leq c\}$ and $U$ is a neighborhood of the critical point $u$.
The multiplicity of $u$ is the number $\mathcal I_{1}(u)$.

It is known that  for a non-degenerate critical point $u$
(that is, the selfadjoint operator associated to $I''(u)$
is an isomorphism)
it is $\mathcal I_{t}(u)=t^{\mathfrak m (u)}$,
where $\mathfrak m(u)$ is the {\sl (numerical) Morse index of $u$}: the maximal dimension
of the subspaces where $I''(u)[\cdot,\cdot]$ is negative definite.

\subsection{Proof of Theorem \ref{gioGae2}}

First note that $I_{\varepsilon}$ is of class $C^{2}$ and
for $u,v,w\in W_{\varepsilon}$
\begin{equation*}\label{I''}
I_{\varepsilon}''(u)[v,w]=\int_{\mathbb R^{N}}(-\Delta)^{s/2}v(-\Delta)^{s/2}w+\int_{\mathbb R^{N}}V(\varepsilon x)vw-\int_{\mathbb R^{N}}f'(u)vw
\end{equation*}
hence $I_{\varepsilon}''(u)$ is represented by the operator
\begin{equation*}
L_{\varepsilon}(u):=\textrm R(u)-\mathrm K(u): W_{\varepsilon}\to W_{\varepsilon}'
\end{equation*}
where $\mathrm R(u)$ is the Riesz isomorphism
and $\mathrm K(u)$ is compact. Indeed let $v_{n}\rightharpoonup 0$ and $w\in W_{\varepsilon}$;
given $\xi > 0$, by \eqref{f_{2}} and \eqref{f_{3}}, for some constant $C_{\xi}>0$
we have
$$\int_{\mathbb R^{N}}\Big|f'(u)v_{n}w\Big|\leq \xi \int_{\mathbb R^{N}} |v_{n}w| +C_{\xi}\int_{\mathbb R^{N}}|u|^{q-1}|v_{n}w|$$
and using that  $v_{n}\rightharpoonup 0$ and the fact that $\xi$ is arbitrary, we  deduce
$$\|\mathrm K(u)[v_{n}]\|=\sup_{\|w\|_{W_{\varepsilon}}=1}\Big|\int_{\mathbb R^{N}}f'(u)v_{n}w\Big|\rightarrow 0.$$

Now  for $a\in(0,+\infty]$, let
$$I_{\varepsilon}^{a}:=\Big\{u\in W_{\varepsilon}: I_{\varepsilon}(u)\leq a\Big\}\ , \qquad 
\mathcal N_{\varepsilon}^{a}:= \mathcal N_{\varepsilon}\cap  I_{\varepsilon}^{a}$$\smallskip
$$\mathcal K_{\varepsilon}:=\Big\{u\in W_{\varepsilon}: I'_{\varepsilon}(u)=0\Big\}\ ,\qquad
\mathcal K_{\varepsilon}^{a}:= \mathcal K_{\varepsilon}\cap  I_{\varepsilon}^{a}\ ,\qquad
(\mathcal K_{\varepsilon})_{a}:=\Big\{u\in \mathcal K_{\varepsilon}: I_{\varepsilon}(u)> a\Big\} .
$$
In the remaining part of this section we will follow \cite{BC,claudianorserginhorodrigo}.
Let $\varepsilon^{*}>0$ small as at the end of Section \ref{Bary} 
and let $\varepsilon\in (0,\varepsilon^{*}]$ be fixed.
In particular $I_{\varepsilon}$ satisfies the Palais-Smale condition.
 We are going to prove that $I_{\varepsilon}$ restricted to $\mathcal N_{\varepsilon}$
 has at least $2\mathcal P_{1}(M)-1$ critical points (for small $\varepsilon$).
 Then Theorem \ref{gioGae2} will follow by Corollary \ref{C1}.
 \medskip
 
 We can assume,
 of course,  that there exists a regular value $b^{*}_{\varepsilon}>m(V_{0})$
 for the functional $I_{\varepsilon}$. Moreover, possibly reducing $\varepsilon^{*}$, we can assume that, see \eqref{subnehari},
 $$\Phi_{\varepsilon}: M\to  \mathcal N^{m(V_{0})+h(\varepsilon)}_{\varepsilon}\subset \mathcal N_{\varepsilon}^{b_{\varepsilon}^{*}}.$$

Since $\Phi_{\varepsilon}$ is injective, it induces
injective homomorphisms in the homology groups, then $\dim H_{k}(M)
\leq \dim H_{k}(\mathcal N_{\varepsilon}^{b_{\varepsilon}^{*}})$
and consequently
\begin{equation}\label{obvious1}
\mathcal P_{t}(\mathcal N_{\varepsilon}^{b_{\varepsilon}^{*}})=\mathcal P_{t}(M)+\mathcal Q(t), \qquad \mathcal Q\in \mathbb P,
\end{equation}
where hereafter $\mathbb P$ denotes the set of polynomials with non-negative integer coefficients.
 
 \medskip

The following result is analogous to \cite[Lemma 5.2]{BC}; we omit the proof.
\begin{lem}
Let  $r\in (0,m(V_{0}))$  and $a\in(r,+\infty]$ a regular level for $I_{\varepsilon}$.
Then
 \begin{eqnarray}\label{PtP}
\mathcal P_{t}(I_{\varepsilon}^{a},I_{\varepsilon}^{r})&=&t \mathcal P_{t}(\mathcal N_{\varepsilon}^{a}).
\end{eqnarray}
\end{lem}

\medskip

In particular we have the following
\begin{coro}\label{t}
Let $r\in (0,m(V_{0}))$. Then
\begin{eqnarray*}
\mathcal P_{t}(I_{\varepsilon}^{b_{\varepsilon}^{*}},I_{\varepsilon}^{r})&=&t\Big(\mathcal P_{t}(M)+\mathcal Q(t)\Big), \qquad \mathcal Q\in \mathbb P, \label{prima}\\
\mathcal P_{t}(W_{\varepsilon}, I_{\varepsilon}^{r})&=& 
t.\label{seconda}
\end{eqnarray*}
\end{coro}
\begin{proof}
The first identity follows by    \eqref{obvious1} and \eqref{PtP} by choosing  $a=b^{*}_{\varepsilon}$.
The second one follows by \eqref{PtP} with $a=+\infty$ and
noticing that the Nehari manifold $\mathcal N_{\varepsilon}$ is contractible. 
\end{proof}

To deal with critical points above the level $b^{*}_{\varepsilon},$ we need also the following
\begin{lem}\label{quadrato}
It holds
$$\mathcal P_{t}(W_{\varepsilon},I_{\varepsilon}^{b_{\varepsilon}^{*}})=t^{2}\Big(\mathcal P_{t}(M)+\mathcal Q(t)-1\Big), \qquad \mathcal Q\in \mathbb P.$$
\end{lem}
\begin{proof}
The proof is purely algebraic and goes exactly as in \cite[Lemma 5.6]{BC}, see also \cite[Lemma 2.4]{claudianorserginhorodrigo}.
\end{proof}

As a consequence of these facts we have

\begin{coro}\label{separato}
Suppose that 
the set $\mathcal K_{\varepsilon}$
is discrete. Then
\begin{eqnarray*}\label{1}
\sum_{u\in \mathcal K^{b^{*}}_{\varepsilon}}\mathcal I_{t}(u)=t\Big(\mathcal P_{t}(M)+\mathcal Q(t)\Big)+(1+t)\mathcal Q_{1}(t)
\end{eqnarray*}
and 
\begin{eqnarray*}
\sum_{u\in(\mathcal K_{\varepsilon})_{b^{*}}} \mathcal I_{t}(u)=t^{2}\Big(\mathcal P_{t}(M)+\mathcal Q(t)-1\Big)+(1+t)\mathcal Q_{2}(t),
\end{eqnarray*}
where $\mathcal Q,\mathcal Q_{1}, \mathcal Q_{2}\in \mathbb P.$

\end{coro}
\begin{proof}
Indeed the Morse theory 
gives
\begin{eqnarray*}\label{1}
\sum_{u\in \mathcal K^{b_{\varepsilon}^{*}}_{\varepsilon}}\mathcal I_{t}(u)=\mathcal P_{t}(I_{\varepsilon}^{b_{\varepsilon}^{*}}, I_{\varepsilon}^{r})+(1+t)\mathcal Q_{1}(t)\\
\end{eqnarray*}
and 
\begin{eqnarray*}
\sum_{u\in(\mathcal K_{\varepsilon})_{b_{\varepsilon}^{*}}} \mathcal I_{t}(u)=
\mathcal P_{t}(W_{\varepsilon},I_{\varepsilon}^{b_{\varepsilon}^{*}})+(1+t)\mathcal Q_{2}(t)
\end{eqnarray*}
so that, by using  Corollary \ref{t} and Lemma \ref{quadrato}, we easily conclude
\end{proof}

Finally, by Corollary \ref{separato} we get
$$\sum_{u\in \mathcal K_{\varepsilon}}\mathcal I_{t}(u)=t\mathcal P_{t}(M)+t^{2}\Big(\mathcal P_{t}(M)-1\Big)+t(1+t)\mathcal Q(t)$$
for some $\mathcal Q\in \mathbb P.$ 
We easily deduce that, if the critical points of $I_{\varepsilon}$
are non-degenerate, then they are  at least $2\mathcal P_{1}(M)-1$,
if counted with their multiplicity.

The proof of Theorem \ref{gioGae2} is thereby complete.
\medskip



\end{document}